\documentclass[reqno]{amsart}
\usepackage{amssymb,setspace}
\usepackage[hmarginratio=1:1,textwidth=17cm,height=23cm]{geometry}%
\usepackage{ifpdf}
\ifpdf
 \usepackage[hyperindex,pagebackref]{hyperref}%
\else
 \expandafter\ifx\csname dvipdfm\endcsname\relax
 \usepackage[hypertex,hyperindex,pagebackref]{hyperref}%
 \else
 \usepackage[dvipdfm,hyperindex,pagebackref]{hyperref}%
 \fi
\fi
\allowdisplaybreaks[4]
\numberwithin{equation}{section}
\theoremstyle{plain}
\newtheorem{thm}{Theorem}[section]
\newtheorem{lem}{Lemma}[section]
\newtheorem{open}{Open Problem}[section]
\newtheorem{prop}{Proposition}[section]
\theoremstyle{plain}
\newtheorem{defn}{Definition}[section]
\theoremstyle{remark}
\newtheorem{rem}{Remark}[section]
\DeclareMathOperator{\td}{d}
\newcommand{\cmdeg}[1]{\sideset{}{{}_\mathrm{cm}^{#1}}\deg}

\begin{document}

\title[Completely monotonic degree of a function]
{Completely monotonic degree of a function involving the tri- and tetra-gamma functions}

\author[F. Qi]{Feng Qi}
\address{Department of Mathematics, College of Science, Tianjin Polytechnic University, Tianjin City, 300387, China; Institute of Mathematics, Henan Polytechnic University, Jiaozuo City, Henan Province, 454010, China}
\email{\href{mailto: F. Qi <qifeng618@gmail.com>}{qifeng618@gmail.com}, \href{mailto: F. Qi <qifeng618@hotmail.com>}{qifeng618@hotmail.com}, \href{mailto: F. Qi <qifeng618@qq.com>}{qifeng618@qq.com}}
\urladdr{\url{http://qifeng618.wordpress.com}}

\begin{abstract}
Let $\psi(x)$ be the di-gamma function, the logarithmic derivative of the classical Euler's gamma function $\Gamma(x)$.
In the paper, the author shows that the completely monotonic degree of the function $[\psi'(x)]^2+\psi''(x)$ is $4$, surveys the history and motivation of the topic, supplies a proof for the claim that a function $f(x)$ is strongly completely monotonic if and only if the function $xf(x)$ is completely monotonic, conjectures the completely monotonic degree of a function involving $[\psi'(x)]^2+\psi''(x)$, presents the logarithmic concavity and monotonicity of an elementary function, and poses an open problem on convolution of logarithmically concave functions.
\end{abstract}



\thanks{This paper was typeset using \AmS-\LaTeX}

\maketitle

\section{Introduction}

A function $f$ is said to be completely monotonic on an interval $I$ if $f$ has derivatives of all
orders on $I$ and
$
0\le(-1)^{k-1}f^{(k-1)}(x)<\infty
$
for $x\in I$ and $k\in\mathbb{N}$, where $f^{(0)}(x)$ means $f(x)$ and $\mathbb{N}$ is the set of all positive integers. See~~\cite[Chapter~XIII]{mpf-1993},~\cite[Chapter~1]{Schilling-Song-Vondracek-2nd}, and~\cite[Chapter~IV]{widder}.
\begin{defn}
Let $f(x)$ be a completely monotonic function on $(0,\infty)$ and denote $f(\infty)=\lim_{x\to\infty}f(x)$. If for some $r\in\mathbb{R}$ the function $x^r[f(x)-f(\infty)]$ is completely monotonic on $(0,\infty)$ but $x^{r+\varepsilon}[f(x)-f(\infty)]$ is not for any positive number $\varepsilon>0$, then we say that the number $r$ is the completely monotonic degree of $f(x)$ with respect to $x\in(0,\infty)$; if for all $r\in\mathbb{R}$ each and every $x^r[f(x)-f(\infty)]$ is completely monotonic on $(0,\infty)$, then we say that the completely monotonic degree of $f(x)$ with respect to $x\in(0,\infty)$ is $\infty$.
\end{defn}
For convenience, we use the notation $\cmdeg{x}[f(x)]$ to denote the completely monotonic degree $r$ of $f(x)$ with respect to $x\in(0,\infty)$. It is clear that the completely monotonic degree $\cmdeg{x}[f(x)]$ of any completely monotonic function $f(x)$ on $(0,\infty)$ is at leat $0$.
It was proved in~\cite[Section~1.5]{Bessel-ineq-Dgree-CM.tex} that the completely monotonic degree $\cmdeg{x}[f(x)]$ equals $\infty$ if and only if $f(x)$ is nonnegative and identically constant. This definition slightly modifies the corresponding one stated in~\cite{psi-proper-fraction-degree-two.tex, Bessel-ineq-Dgree-CM.tex, simp-exp-degree-ext.tex} and related references therein. For simplicity, in what follows, we sometimes just say that $\cmdeg{x}[f(x)]$ is the completely monotonic degree of $f(x)$.
\par
The classical Euler's gamma function $\Gamma(x)$ may be defined for $x>0$ by
\begin{equation*}%\label{gamma-dfn}
\Gamma(x)=\int^\infty_0t^{x-1} e^{-t}\td t.
\end{equation*}
The logarithmic derivative of $\Gamma(x)$, denoted by $\psi(x)=\frac{\Gamma'(x)}{\Gamma(x)}$, is called the psi or di-gamma function, the derivatives $\psi'(x)$ and $\psi''(x)$ are respectively called the tri- and tetra-gamma functions. As a whole, hereafter, the derivatives $\psi^{(i)}(x)$ for $i\ge0$ are called polygamma functions.
\par
The purpose of this paper is to compute the completely monotonic degree of the function
\begin{equation}\label{psi-der-square-psi-2der}
\Psi(x)=[\psi'(x)]^2+\psi''(x)
\end{equation}
with respect to $x\in(0,\infty)$. We may state our main result as the following theorem.

\begin{thm}\label{x-4-further-thm}
The completely monotonic degree of $\Psi(x)$ defined by~\eqref{psi-der-square-psi-2der} with respect to $x\in(0,\infty)$ is $4$, that is,
\begin{equation}\label{degree=4}
\cmdeg{x}[\Psi(x)]=4.
\end{equation}
\end{thm}

For proving Theorem~\ref{x-4-further-thm} in Section~\ref{sec-2-x-4-di-tri-gamma-furth}, we will prepare some lemmas in Section~\ref{lemmas-sec}. In the last section, we shall survey the history and motivation of this topic, conjecture the completely monotonic degree of a function involving $\Psi(x)$, supply a proof for the claim that a function $f(x)$ is strongly completely monotonic if and only if the function $xf(x)$ is completely monotonic, present the logarithmic concavity and monotonicity of an elementary function, and pose an open problem on convolution of logarithmically concave functions.

\section{Lemmas}\label{lemmas-sec}

In order to smoothly prove Theorem~\ref{x-4-further-thm}, we need the following lemmas.

\begin{lem}[{\cite[p.~260, 6.4.1]{abram}}]
For $n\in\mathbb{N}$ and $x>0$,
\begin{equation}\label{psin}
\psi^{(n)}(x)=(-1)^{n+1}\int_{0}^{\infty}\frac{t^{n}}{1-e^{-t}}e^{-xt}\td t.
\end{equation}
\end{lem}

\begin{lem}[\cite{abram, widder}]
Let $f_i(t)$ for $i=1,2$ be piecewise continuous in arbitrary finite intervals included in $(0,\infty)$ and suppose that there exist some constants $M_i>0$ and $c_i\ge0$ such that $\vert f_i(t)\vert\le M_ie^{c_it}$ for $i=1,2$. Then
\begin{equation}\label{convlotion}
\int_0^\infty\biggl[\int_0^tf_1(u)f_2(t-u)\td u\biggr]e^{-st}\td t
=\int_0^\infty f_1(u)e^{-su}\td u\int_0^\infty f_2(v)e^{-sv}\td v.
\end{equation}
\end{lem}

\begin{lem}[{\cite[p.~631, Lemma~2.1]{qi-deb-evaluation.tex}}]\label{l1}
Let $f(x,t)$ is differentiable in $t$ and continuous for $(x,t)\in \mathbb{R}^2$. Then
\begin{equation}\label{eq5}
\frac \td{\td t}\int_{x_0}^tf(x,t)\td x=f(t,t)+\int_{x_0}^t\frac {\partial {f(x,t)}}{\partial t}\td x.
\end{equation}
\end{lem}

\begin{lem}[{\cite[Theorem~2]{Beesack-Beograd-1969}, \cite[p.~325]{Imoru-JMAA-1992}, and~\cite[p.~374]{mit}}]\label{convolution-ineq}
If $f_i$ for $1\le i\le n$ are nonnegative Lebesgue square integrable functions on $[0,a)$ for all $a>0$, then for all $n\ge2$ and $x\ge0$,
\begin{equation}\label{convolution-lower=b}
 f_1\ast\dotsm\ast f_n(x)\ge\frac{x^{n-1}}{(n-1)!} \exp\Biggl[\frac{n-1}{x^{n-1}}\int_0^x(x-u)^{n-2}\sum_{j=1}^n\ln f_j(u)\td u\Biggr],
\end{equation}
where $f_i\ast f_j(x)$ denotes the convolution $\int_0^xf_i(t)f_j(x-t)\td t$.
\end{lem}

\begin{lem}[{\cite[p.~260, 6.4.12, 6.4.13 and 6.4.14]{abram}}] \label{three-asymptotic-formula}
As $z\to\infty$ in $|\arg z|<\pi$,
\begin{align}
\psi'(z)&\sim\frac1z+\frac1{2z^2}+\frac1{6z^3}-\frac1{30z^5}+\frac1{42z^7}-\frac1{30z^9}+\dotsm,\\
\psi''(z)&\sim-\frac1{z^2}-\frac1{z^3}-\frac1{2z^4}+\frac1{6z^6}-\frac1{6z^8}+\frac3{10z^{10}} -\frac5{6z^{12}}+\dotsm,\\
\psi^{(3)}(z)&\sim\frac2{z^3}+\frac3{z^4}+\frac2{z^5}-\frac1{z^7}+\frac4{3z^9} -\frac3{z^{11}}+\frac{10}{z^{13}}-\dotsm.
\end{align}
\end{lem}

The formulas listed in Lemma~\ref{three-asymptotic-formula} are special cases of~\cite[p.~260, 6.4.11]{abram}.

\section{Proof of Theorem~\ref{x-4-further-thm}}\label{sec-2-x-4-di-tri-gamma-furth}

Now we are in a position to compute the completely monotonic degree of the function $\Psi(x)$.
\par
Using the integral representation~\eqref{psin} and the formula~\eqref{convlotion} gives
\begin{align*}
\Psi(x)&=\biggl[\int_{0}^{\infty}\frac{t}{1-e^{-t}}e^{-xt}\td t\biggr]^2 -\int_{0}^{\infty}\frac{t^2}{1-e^{-t}}e^{-xt}\td t\\
&=\int_{0}^{\infty}\biggl[\int_0^t\frac{s(t-s)}{(1-e^{-s})\bigl[1-e^{-(t-s)}\bigr]}\td s
-\frac{t^2}{1-e^{-t}}\biggr]e^{-xt}\td t\\
&=\int_{0}^{\infty}q(t)e^{-xt}\td t,
\end{align*}
where
\begin{equation}\label{sigma(s)=dfn}
q(t)=\int_0^t\sigma(s)\sigma(t-s)\td s-t\sigma(t)\quad \text{and}\quad 
\sigma(s)=
\begin{cases}
\dfrac{s}{1-e^{-s}},& s\ne0\\
1,& s=0.
\end{cases}
\end{equation}
\par
Direct calculations reveal
\begin{align*}
 \sigma'(s)&=1+\frac{1-s}{e^s-1}-\frac{s}{(e^s-1)^2},\quad
 \sigma''(s)=\frac{s-2}{e^s-1}+\frac{3 s-2}{(e^s-1)^2}+\frac{2 s}{(e^s-1)^3},\\
 \sigma^{(3)}(s)&=\frac{3-s}{e^s-1}+\frac{9-7 s}{(e^s-1)^2}-\frac{6 (2 s-1)}{(e^s-1)^3}-\frac{6 s}{(e^s-1)^4},\\
 \sigma^{(4)}(s)&=\frac{s-4}{e^s-1}+\frac{15 s-28}{(e^s-1)^2}+\frac{2 (25 s-24)}{(e^s-1)^3}+\frac{12 (5 s-2)}{(e^s-1)^4}+\frac{24 s}{(e^s-1)^5},\\
 \sigma^{(5)}(s)&=\frac{5-s}{e^s-1}+\frac{75-31 s}{(e^s-1)^2}-\frac{10 (18 s-25)}{(e^s-1)^3}-\frac{30 (13 s-10)}{(e^s-1)^4}-\frac{120 (3 s-1)}{(e^s-1)^5}-\frac{120 s}{(e^s-1)^6},\\
 \sigma^{(6)}(s)&=\frac{s-6}{e^s-1}+\frac{3 (21 s-62)}{(e^s-1)^2}+\frac{2 (301 s-540)}{(e^s-1)^3}+\frac{60 (35 s-39)}{(e^s-1)^4}+\frac{240 (14 s-9)}{(e^s-1)^5}+\frac{360 (7 s-2)}{(e^s-1)^6}+\frac{720 s}{(e^s-1)^7},
\end{align*}
and
\begin{gather*}
\sigma(0)=1, \quad \sigma'(0)=\frac12,\quad \sigma''(0)=\frac16, \quad \sigma^{(3)}(0)=0, \quad \sigma^{(4)}(0)=-\frac1{30},\quad \sigma^{(5)}(0)=0, \quad \sigma^{(6)}(0)=\frac1{42}.
\end{gather*}
Further differentiating consecutively brings out
\begin{align*}
 [\ln\sigma''(s)]'&=-\frac{(s-3)e^{2 s} +4s e^s +s+3}{[(s-2)e^s+s+2](e^s-1)},\\
 [\ln\sigma''(s)]''&= -\frac{e^{4s}-4(s^2-3s+4)e^{3 s}-(4 s^2-30)e^{2 s}-4(s^2+3s+4)e^s+1}{(e^s-1)^2[(s-2)e^s+s+2]^2}\\
 &\triangleq-\frac{h_1(s)}{(e^s-1)^2[(s-2)e^s+s+2]^2},\\
 h_1'(s)&=4[e^{3 s}-(3 s^2-7 s+9)e^{2 s}-(2 s^2+2s-15)e^s-s^2-5 s-7]e^s\\
 &\triangleq4h_2(s)e^s,\\
 h_2'(s)&=3 e^{3 s}- (6s^2-8 s+11)e^{2 s}-(2 s^2+6 s-13)e^s -2s-5,\\
 h_2''(s)&=9e^{3 s}-2 (6 s^2-2 s+7) e^{2 s}- (2 s^2+10 s-7)e^s-2,\\
 h_2^{(3)}(s)&=[27 e^{2s}-8 e^s (3 s^2+2 s+3)-2 s^2-14 s-3]e^s\\
 &\triangleq h_3(s)e^s,\\
 h_3'(s)&=54 e^{2 s}-8 (3 s^2+8 s+5) e^s-2 (2s+7),\\
 h_3''(s)&=4[27 e^{2 s}-2 (3 s^2+14 s+13) e^s-1],\\
 h_3^{(3)}(s)&=8 (27 e^s-3 s^2-20 s-27)e^s \\
 &>0
\end{align*}
for $s\in(0,\infty)$, and\label{h(x)-positivity-proof}
\begin{equation*}
 h_3''(0)=h_3'(0)=h_3(0)=h_2^{(3)}(0)=h_2''(0)=h_2'(0)=h_2(0)=h_1'(0)=h_1(0)=0.
\end{equation*}
This means that
\begin{gather*}
 h_3''(s)>0,\quad h_3'(s)>0,\quad h_3(s)>0,\quad h_2^{(3)}(s)>0,\\
 h_2''(s)>0,\quad h_2'(s)>0,\quad h_2(s)>0,\quad h_1'(s)>0,\quad h_1(s)>0
\end{gather*}
for $s\in(0,\infty)$.
Therefore, the derivative $[\ln\sigma''(s)]''$ is negative, that is, the function $\sigma''(s)$ is logarithmically concave, on $(0,\infty)$. Hence, for any given number $t>0$,
\begin{enumerate}
\item
the function $\sigma''(s)\sigma''(t-s)$ is also logarithmically concave with respect to $s\in(0,t)$;
\item
the function $\sigma''(s)$ is decreasing and $\sigma(s)$ is not concave on $(0,\infty)$.
\end{enumerate}
\par
By Lemma~\ref{l1} and integration-by-part, straightforward computations yield
\begin{align*}
 q'(t)&=\int_0^t\sigma(s)\sigma'(t-s)\td s+\sigma(0)\sigma(t)-[t\sigma'(t)+\sigma(t)]\\
 &=\int_0^t\sigma(s)\sigma'(t-s)\td s-t\sigma'(t),\\
 q''(t)&=\int_0^t\sigma(s)\sigma''(t-s)\td s+\sigma(t)\sigma'(0)-[\sigma'(t)+t\sigma''(t)]\\
 &=-\int_0^t\sigma(s)\frac{\td \sigma'(t-s)}{\td s}\td s+\sigma(t)\sigma'(0)-[\sigma'(t)+t\sigma''(t)]\\
 &=\int_0^t\sigma'(s)\sigma'(t-s)\td s-t\sigma''(t),\\
 q^{(3)}(t)&=\int_0^t\sigma'(s)\sigma''(t-s)\td s +\frac12\sigma'(t)-\sigma''(t)-t\sigma^{(3)}(t),\\
 q^{(4)}(t)&=\int_0^t\sigma'(s)\sigma^{(3)}(t-s)\td s+\frac16\sigma'(t) +\frac12\sigma''(t)-2\sigma^{(3)}(t)-t\sigma^{(4)}(t)\\
&=-\int_0^t\sigma'(s)\frac{\td \sigma''(t-s)}{\td s}\td s+\frac16\sigma'(t) +\frac12\sigma''(t)-2\sigma^{(3)}(t)-t\sigma^{(4)}(t)\\
&=\int_0^t\sigma''(s)\sigma''(t-s)\td s +\sigma''(t)-2\sigma^{(3)}(t)-t\sigma^{(4)}(t)\\
&=2\int_0^{t/2}\sigma''(s)\sigma''(t-s)\td s +\sigma''(t)-2\sigma^{(3)}(t)-t\sigma^{(4)}(t),
\end{align*}
and
\begin{equation*}
q(0)=q'(0)=q''(0)=0, \quad q^{(3)}(0)=\frac1{12}, \quad q^{(4)}(0)=\frac16.
\end{equation*}
\par
Applying Lemma~\ref{convolution-ineq} to $f_1=f_2=\sigma''$ and $n=2$ leads to
\begin{equation*}
\int_0^t\sigma''(s)\sigma''(t-s)\td s
\ge t \exp\biggl[\frac2t\int_0^t\ln \sigma''(u)\td u\biggr].
\end{equation*}
Hence, the validity of the inequality
\begin{equation}\label{final-suffienct-ineq}
t \exp\biggl[\frac2t\int_0^t\ln \sigma''(u)\td u\biggr] +\sigma''(t)-2\sigma^{(3)}(t)-t\sigma^{(4)}(t)>0
\end{equation}
implies the positivity of $q^{(4)}(t)$ on $(0,\infty)$.
\par
When $t\sigma^{(4)}(t)+2\sigma^{(3)}(t)-\sigma''(t)\le0$, the inequality~\eqref{final-suffienct-ineq} is clearly valid.
\par
When $t\sigma^{(4)}(t)+2\sigma^{(3)}(t)-\sigma''(t)>0$, the inequality~\eqref{final-suffienct-ineq} may be rearranged as
\begin{equation*}
\int_0^t\ln \sigma''(u)\td u
>\frac{t}2\ln\frac{t\sigma^{(4)}(t)+2\sigma^{(3)}(t)-\sigma''(t)}t.
\end{equation*}
Let
\begin{equation}
F(t)=\int_0^t\ln \sigma''(u)\td u-\frac{t}2\ln\frac{t\sigma^{(4)}(t)+2\sigma^{(3)}(t)-\sigma''(t)}t.
\end{equation}
Differentiating twice produces
\begin{equation*}
F'(t)=\ln \sigma''(t)-\frac12\ln\frac{t\sigma^{(4)}(t)+2\sigma^{(3)}(t)-\sigma''(t)}t
-\frac{t^2\sigma^{(5)}(t)+2t\sigma^{(4)}(t)-(t+2)\sigma^{(3)}(t)+\sigma''(t)} {2[t\sigma^{(4)}(t)+2\sigma^{(3)}(t)-\sigma''(t)]}
\end{equation*}
and
\begin{align*}
F''(t)&=\frac{\sigma^{(3)}(t)}{\sigma''(t)}-\frac{t^2\sigma^{(5)}(t) +2t\sigma^{(4)}(t)-(t+2)\sigma^{(3)}(t)+\sigma''(t)} {2t[t\sigma^{(4)}(t)+2\sigma^{(3)}(t)-\sigma''(t)]} -\frac1{2[t\sigma^{(4)}(t)+2\sigma^{(3)}(t)-\sigma''(t)]^2}\\
&\quad\times \bigl\{[t^2\sigma^{(6)}(t)+4t\sigma^{(5)}(t)-t\sigma^{(4)}(t)] [t\sigma^{(4)}(t)+2\sigma^{(3)}(t)-\sigma''(t)] -[t^2\sigma^{(5)}(t)+2t\sigma^{(4)}(t)\\
&\quad-(t+2)\sigma^{(3)}(t)+\sigma''(t)] [t\sigma^{(5)}(t)+3\sigma^{(4)}(t)-\sigma^{(3)}(t)]\bigr\}\\
&\triangleq\frac{Q(t)}{2t\sigma''(t)[t\sigma^{(4)}(t)+2\sigma^{(3)}(t)-\sigma''(t)]^2},
\end{align*}
where
\begin{align*}
Q(t)&=2t\sigma^{(3)}(t)\bigl[t\sigma^{(4)}(t)+2\sigma^{(3)}(t)-\sigma''(t)\bigr]^2 -\sigma''(t)\bigl[t\sigma^{(4)}(t)+2\sigma^{(3)}(t) -\sigma''(t)\bigr]\bigl[t^2\sigma^{(5)}(t) +2t\sigma^{(4)}(t)\\
&\quad-(t+2)\sigma^{(3)}(t)+\sigma''(t)\bigr]-t\sigma''(t)\bigl\{\bigl[t^2\sigma^{(6)}(t)+4t\sigma^{(5)}(t) -t\sigma^{(4)}(t)\bigr] \bigl[t\sigma^{(4)}(t)+2\sigma^{(3)}(t)-\sigma''(t)\bigr] \\
&\quad -\bigl[t^2\sigma^{(5)}(t)+2t\sigma^{(4)}(t)-(t+2)\sigma^{(3)}(t)+\sigma''(t)\bigr] \bigl[t\sigma^{(5)}(t)+3\sigma^{(4)}(t)-\sigma^{(3)}(t)\bigr]\bigr\}\\
&\triangleq \frac{e^{3 t}R(t)}{(e^t-1)^{15}}
\end{align*}
and
\begin{equation}\label{R(t)-dfn-eq}
\begin{split}
R(t)&=e^{9 t} \bigl(t^5-12 t^4+70 t^3-160 t^2+192 t-128\bigr)-e^{8 t} \bigl(16 t^7-220 t^6+1219 t^5-3220 t^4+4490 t^3\\
&\quad-3248 t^2+1152 t-768\bigr)-4 e^{7 t} \bigl(37 t^7-423 t^6+1397 t^5-1409 t^4-1020 t^3+2632 t^2-732 t\\
&\quad+456\bigr)-4 e^{6 t} \bigl(225 t^7-1281 t^6+1213 t^5+3127 t^4-4372 t^3-2648 t^2+1020 t-504\bigr)\\
&\quad-2 e^{5 t} \bigl(908 t^7-1514 t^6-6493 t^5+8710 t^4+12754 t^3-1216 t^2-1656 t+336\bigr)-2 e^{4 t} \bigl(908 t^7\\
&\quad+1710 t^6-5489 t^5-12370 t^4+594 t^3+4880 t^2+696 t+336\bigr)-4 e^{3 t} \bigl(225 t^7+1263 t^6\\
&\quad+1771 t^5-887 t^4-3208 t^3-728 t^2+12 t-168\bigr)-4 e^{2 t} \bigl(37 t^7+353 t^6+1099 t^5+1337 t^4\\
&\quad+272 t^3-632 t^2-108 t+24\bigr)-e^t \bigl(16 t^7+180 t^6+827 t^5+1864 t^4+2226 t^3+1312 t^2+240 t\\
&\quad+96\bigr)+t^5+8 t^4+30 t^3+48 t^2+48 t+32.
\end{split}
\end{equation}
Differentiating and taking the limit $t\to0$ about $76$ times respectively by the same approach as either the proof of the positivity of $\theta(t)$ in~\cite[pp.~472\nobreakdash--476]{x-4-di-tri-gamma-p(x)-Slovaca.tex}, or proofs of the absolute monotonicity of the functions $f_1,f_2,f_3$ and $h_1,h_2,h_3,h_4$ in~\cite[Theorem~1.1]{Mortici-monoburn.tex}, or the proof of the positivity of $h_1(s)$ on page~\pageref{h(x)-positivity-proof} in this paper, we may verify the positivity of $R(t)$ on $(0,\infty)$. In Section~\ref{appendix} below, we will prove the positivity of $R(t)$ on $(0,\infty)$ in details. This means that $Q(t)>0$ on $(0,\infty)$ and $F''(t)>0$. Accordingly, the derivative $F'(t)$ is strictly increasing. Because
\begin{align*}
F'(8)&=4+\frac{3 (6 e^{32}+729 e^{24}+2825 e^{16}+1483 e^8+77)}
{8 e^{32}+270 e^{24}+150 e^{16}-374 e^8-54}+\frac12\ln \frac{8(5+3 e^8)}{(e^8-1)(27+214 e^8+139 e^{16}+4 e^{24})}\\
&=-0.24428\dotsc
\end{align*}
and
\begin{align*}
F'(10)&=5+\frac{72 e^{40}+4715 e^{30}+16563 e^{20}+8241 e^{10}+409}
{19 e^{40}+440 e^{30}+186 e^{20}-568 e^{10}-77}+\frac12\ln\frac{80(3+2 e^{10})^2}
{(e^{10}-1)(77+645 e^{10}+459 e^{20}+19 e^{30})}\\
&=0.20823\dotsc,
\end{align*}
which are numerically calculated with the help of M\textsc{athematica},
the unique zero of $F'(t)$ locates on the open interval $(8,10)$. Consequently, the unique minimum of the function $F(t)$ attains on the interval $(8,10)$. Since
\begin{equation*}
F(t)=F(t_0)+(t-t_0)F'(t_0)+\frac{(t-t_0)^2}2F''(\xi)>F(t_0)+(t-t_0)F'(t_0)
\end{equation*}
for $t,t_0\in[8,10]$, where $\xi$ locates between $t_0$ and $t$, numerically calculating with the help of M\textsc{athematica} gains
\begin{align*}
2F(t)&>[F(8)+(t-8)F'(8)]+[F(10)+(t-10)F'(10)]\\
&=F(8)+F(10)-[8F'(8)+10F'(10)]+[F'(8)+F'(10)]t\\
&>\int_0^8 \ln\sigma''(u)\td u-4 \ln\frac{e^8 (27+214 e^8+139 e^{16}+4 e^{24})}{2 (e^8-1)^5}\\
&\quad+\int_0^{10} \ln\sigma''(u)\td u-5 \ln\frac{e^{10} (77+645 e^{10}+459 e^{20}+19 e^{30})}{5 (e^{10}-1)^5} -0.1281-0.0361t\\
&>\int_0^8 \ln\sigma''(u)\td u+\int_0^{10} \ln\sigma''(u)\td u+72.492 -0.1281-0.361\\
&>\int_0^8 \ln\sigma''(u)\td u+\int_0^{10} \ln\sigma''(u)\td u+72\\
&>\frac13\Biggl[\sum_{k=1}^{24}\ln\sigma''\biggl(\frac{k}3\biggr) +\sum_{k=1}^{30}\ln\sigma''\biggl(\frac{k}3\biggr)\Biggr]+72\\
&>-29-43+72\\
&=0
\end{align*}
on the interval $[8,10]$. In conclusion, the inequality~\eqref{final-suffienct-ineq} is valid and the fourth derivative $q^{(4)}(t)$ is positive on $(0,\infty)$.
\par
Integrating by parts successively results in
\begin{gather*}
x^4\Psi(x)=x^4\int_0^\infty q(t)e^{-xt}\td t
=-x^3\int_0^\infty q(t)\frac{\td e^{-xt}}{\td t}\td t
=-x^3\biggl[q(t)e^{-xt}\big|_{t=0}^{t=\infty}-\int_0^\infty q'(t)e^{-xt}\td t\biggr]\\
=x^3\int_0^\infty q'(t)e^{-xt}\td t
=x^2\int_0^\infty q''(t)e^{-xt}\td t
=x\int_0^\infty q^{(3)}(t)e^{-xt}\td t\\
=-\int_0^\infty q^{(3)}(t)\frac{\td e^{-xt}}{\td t}\td t
=-\biggl[q^{(3)}(t)e^{-xt}\big|_{t=0}^{t=\infty} -\int_0^\infty q^{(4)}(t)\frac{\td e^{-xt}}{\td t}\td t\biggr]
=\frac1{12}+\int_0^\infty q^{(4)}(t)e^{-xt}\td t.
\end{gather*}
From the positivity of $q^{(4)}(t)$ on $(0,\infty)$, it follows that the function $x^4\Psi(x)$ is completely monotonic on $(0,\infty)$. In other words,
\begin{equation}\label{degree>4}
\cmdeg{x}[\Psi(x)]\ge4.
\end{equation}
\par
Suppose that the function
\begin{equation}\label{f-alpha(x)-dfn}
f_\alpha(x)=x^\alpha \Psi(x)
\end{equation}
is completely monotonic on $(0,\infty)$. Then
\begin{equation*}
f_\alpha'(x)=x^{\alpha-1} \bigl\{\alpha \Psi(x) +x\bigl[2\psi'(x)\psi''(x)+\psi^{(3)}(x)\bigr]\bigr\}\le0
\end{equation*}
on $(0,\infty)$, that is,
\begin{equation}\label{alpha-necessary-ineq}
\alpha \le-\frac{x\bigl[2\psi'(x)\psi''(x)+\psi^{(3)}(x)\bigr]}\Psi(x) \triangleq\phi(x),\quad x>0.
\end{equation}
From Lemma~\ref{three-asymptotic-formula}, it follows
\begin{multline*}
\lim_{x\to\infty}\phi(x)=-\lim_{x\to\infty}\Biggl\{\frac{x} {\bigl[\frac1x+\frac1{2x^2}+O\bigl(\frac1{x^2}\bigr)\bigr]^2
+\bigl[-\frac1{x^2}-\frac1{x^3}+O\bigl(\frac1{x^3}\bigr)\bigr]}\\
\times\biggl[2\biggl(\frac1x+\frac1{2x^2}
+O\biggl(\frac1{x^2}\biggr)\biggr) \biggl(-\frac1{x^2}-\frac1{x^3}+O\biggl(\frac1{x^3}\biggr)\biggr)
+\biggl(\frac2{x^3}+\frac3{x^4}+O\biggl(\frac1{x^4}\biggr)\biggr)\biggr]\Biggr\}
=4.
\end{multline*}
As a result, we have
\begin{equation}\label{degree<4}
\cmdeg{x}[\Psi(x)]\le4.
\end{equation}
\par
Combining~\eqref{degree>4} with~\eqref{degree<4} yields~\eqref{degree=4}. The proof of Theorem~\ref{x-4-further-thm} is complete.

\section{Miscellanea}

In this section, we will survey the history and motivation of this topic, conjecture the completely monotonic degree of a function involving $\Psi(x)$, supply a proof for the assertion that a function $f(x)$ is strongly completely monotonic if and only if the function $xf(x)$ is completely monotonic, mention some results on the function $\sigma''(s)$, and pose an open problem on convolution of logarithmically concave functions.

\subsection{}
We first demonstrate the motivation of this paper by retrospecting the history and by introducing some known results related to Theorem~\ref{x-4-further-thm}.
\par
In~\cite[p.~208, (4.39)]{forum-alzer}, it was established that the inequality
\begin{equation}\label{psi'psi''}
\Psi(x)>\frac{p(x)}{900x^4(x+1)^{10}}
\end{equation}
holds for $x>0$, where
\begin{equation}
p(x)=75x^{10}+900x^9+4840x^8+15370x^7+31865x^6+45050x^5+44101x^4+29700x^3+13290x^2+3600x+450.
\end{equation}
\par
It is clear that the inequality
\begin{equation}\label{positivity}
\Psi(x)>0
\end{equation}
for $x>0$ is a weakened version of the inequality~\eqref{psi'psi''}. This inequality was deduced and recovered in~\cite[Theorem~2.1]{batir-interest} and~\cite[Lemma~1.1]{batir-new}. The inequality~\eqref{positivity} was also employed in~\cite[Theorem~4.3]{alzer-grinshpan},
\cite[Theorem~2.1]{batir-interest}, \cite[Theorem~2.1]{batir-new}, and~\cite[Theorem~2]{Infinite-family-Digamma.tex}. This inequality has been generalized in~\cite[Lemma~4.6]{alzer-grinshpan}, \cite[Lemma~1.2 and~Remark~1.3]{batir-jmaa-06-05-065}, and~\cite{x-4-q-di-tri-gamma.tex, AAM-Qi-09-PolyGamma.tex}. For more information about the history and background of this topic, please refer to the expository and survey articles~\cite{bounds-two-gammas.tex, Wendel-Gautschi-type-ineq-Banach.tex, Wendel2Elezovic.tex-JIA} and plenty of references therein.
\par
In~\cite{BNGuo-FQi-HMSrivastava2010.tex}, it was proved that, among all functions $\bigl[\psi^{(m)}(x)\bigr]^2+\psi^{(n)}(x)$ for $m,n\in\mathbb{N}$, only the function $\Psi(x)$ is nontrivially completely monotonic on $(0,\infty)$.
\par
In~\cite[Theorem~1]{x-4-di-tri-gamma-p(x)-Slovaca.tex} and~\cite[Theorem~2]{x-4-di-tri-gamma-upper-lower-combined.tex}, the functions
\begin{equation}\label{x-4-upper-g(x)}
\frac{x+12}{12x^4(x+1)}-\Psi(x),\quad
\Psi(x)-\frac{x^2+12}{12x^4(x+1)^2},\quad \text{and}\quad
\Psi(x)-\frac{p(x)}{900x^4(x+1)^{10}}
\end{equation}
were proved to be completely monotonic on $(0,\infty)$. From this, we obtain
\begin{equation}\label{psi'psi''-rew}
\max\biggl\{\frac{x^2+12}{12x^4(x+1)^2}, \frac{p(x)}{900x^4(x+1)^{10}}\biggr\}<\Psi(x)
<\frac{x+12}{12x^4(x+1)},\quad x>0.
\end{equation}
\par
In~\cite[Theorem~1]{x-4-di-tri-gamma-upper-improve.tex}, the function
\begin{equation}\label{tetra-square+poly-2}
h_\lambda(x)=\Psi(x)-\frac{x^2+\lambda x+12}{12x^4(x+1)^2}
\end{equation}
was proved to be completely monotonic on $(0,\infty)$ if and only if $\lambda\le0$, and so is $-h_\lambda(x)$ if and only if $\lambda\ge4$; Consequently, the double inequality
\begin{equation}\label{x-4-upper-improve-ineq}
\frac{x^2+\mu x+12}{12x^4(x+1)^2}<\Psi(x)<\frac{x^2+\nu x+12}{12x^4(x+1)^2}
\end{equation}
holds on $(0,\infty)$ if and only if $\mu\le0$ and $\nu\ge4$. The inequality~\eqref{x-4-upper-improve-ineq} refines and sharpens the right-hand side inequality in~\eqref{psi'psi''-rew}.
\par
It was remarked in~\cite[p.~137, Remark~2.7]{Wendel-Gautschi-type-ineq-Banach.tex} that a divided difference version of the inequality~\eqref{positivity} has been implicitly obtained in~\cite{Kazarinoff-56}. The divided difference form of the function $\Psi(x)$ and related functions have been investigated in the papers~\cite{notes-best-simple-open-jkms.tex, BAustMS-5984-RV.tex, notes-best-simple-cpaa.tex, Com-Mon-Di-Tri-Div-simp.tex, SCM-2012-0142.tex} and closely related references therein.
\par
In~\cite[p.~2273, Corollary~3]{Koumandos-MC-08} and~\cite{Koumandos-Lamprecht-MC-2010}, among other things, it was deduced that the functions $f_2(x)$ and $f_3(x)$ defined by~\eqref{f-alpha(x)-dfn} are completely monotonic on $(0,\infty)$. Equivalently, $\cmdeg{x}[\Psi(x)]\ge2$ and $\cmdeg{x}[\Psi(x)]\ge3$.
Motivated by these results, we naturally pose the following problems.
\begin{enumerate}
\item
Is the function $f_4(x)$ defined by~\eqref{f-alpha(x)-dfn} completely monotonic on $(0,\infty)$?
\item
Is $\alpha\le4$ the necessary and sufficient condition for the function $f_\alpha(x)$ in~\eqref{f-alpha(x)-dfn} to be completely monotonic on $(0,\infty)$?
\end{enumerate}
In other words, is the constant $4$ the completely monotonic degree of the function $\Psi(x)$ with respect to $x\in(0,\infty)$?
Theorem~\ref{x-4-further-thm} of this paper affirmatively answers these questions. Hence, Theorem~\ref{x-4-further-thm} strengthens and improves these results and inversely hints and implies that the main results in~\cite{Koumandos-MC-08, Koumandos-Lamprecht-MC-2010} may be further improved, developed, or amended.
\par
Till now, we may further conjecture that the completely monotonic degrees with respect to $x\in(0,\infty)$ of the functions $h_\lambda(x)$ and $-h_\mu(x)$ defined by~\eqref{tetra-square+poly-2} are $4$ if and only if $\lambda\le0$ and $\mu\ge4$. Namely,
\begin{equation}
 \cmdeg{x}[h_\lambda(x)]=\cmdeg{x}[-h_\mu(x)]=4
\end{equation}
if and only if $\lambda\le0$ and $\mu\ge4$.

\subsection{}

Recall from~\cite{Trimble-Wells-Wright} that a function $f$ is said to be strongly completely monotonic on $(0,\infty)$ if it has derivatives of all orders and
$
(-1)^nx^{n+1}f^{(n)}(x)
$
is nonnegative and decreasing on $(0,\infty)$ for all $n\ge0$.
\par
Proposition~1.1 in~\cite[p.~34]{Koumandos-Pedersen-09-JMAA} reads that a function $f(x)$ is strongly completely monotonic if and only if the function $xf(x)$ is completely monotonic. In other words, the set of completely monotonic functions on $(0,\infty)$ of degree not less than $1$ with respect to $x$ coincides with the set of strongly completely monotonic functions on $(0,\infty)$.
\par
Because not finding a proof for~\cite[p.~34, Proposition~1.1]{Koumandos-Pedersen-09-JMAA}, we now provide a proof for it as follows.
If $xf(x)$ is completely monotonic on $(0,\infty)$, then
\begin{equation*}
(-1)^k[xf(x)]^{(k)}=(-1)^k\bigl[xf^{(k)}(x)+kf^{(k-1)}(x)\bigr]
=\frac{(-1)^kx^{k+1}f^{(k)}(x)-k[(-1)^{k-1}x^kf^{(k-1)}(x)]}{x^k}
\ge0
\end{equation*}
on $(0,\infty)$ for all integers $k\ge0$. From this and by induction, we obtain
\begin{gather*}
(-1)^kx^{k+1}f^{(k)}(x)\ge k[(-1)^{k-1}x^kf^{(k-1)}(x)]
\ge k(k-1)[(-1)^{k-2}x^{k-1}f^{(k-2)}(x)]\ge\\
\dotsm\ge [k(k-1)\dotsm4\cdot3]x^3f''(x)
\ge [k(k-1)\dotsm4\cdot3\cdot2]x^2f'(x)\ge k!xf(x)\ge0
\end{gather*}
on $(0,\infty)$ for all integers $k\ge0$. So, the function $f(x)$ is strongly completely monotonic on $(0,\infty)$.
\par
Conversely, if $f(x)$ is a strongly completely monotonic function on $(0,\infty)$, then
\begin{equation*}
 (-1)^kx^{k+1}f^{(k)}(x)\ge0
\end{equation*}
and
\begin{equation*}
\bigl[(-1)^kx^{k+1}f^{(k)}(x)\bigr]'=\frac{(k+1)\bigl[(-1)^kx^{k+1}f^{(k)}(x)\bigr] -(-1)^{k+1}x^{k+2}f^{(k+1)}(x)}x
\le0
\end{equation*}
hold on $(0,\infty)$ for all integers $k\ge0$. Hence, it follows that $xf(x)\ge0$ and $(-1)^{k+1}[xf(x)]^{(k+1)}$ on $(0,\infty)$ for all integers $k\ge0$. As a result, the function $xf(x)$ is completely monotonic on $(0,\infty)$. The proof of~\cite[p.~34, Proposition~1.1]{Koumandos-Pedersen-09-JMAA} is complete.

\subsection{}
The function $\sigma(s)$ defined in~\eqref{sigma(s)=dfn} is a special case of the function
\begin{equation*}
g_{a,b}(s)=
\begin{cases}
 \dfrac{s}{b^{s}-a^{s}},&s\ne0,\\
 \dfrac1{\ln b-\ln a},&s=0,
\end{cases}
\end{equation*}
where $a,b$ are positive numbers and $a\ne b$, and some special cases of the function $g_{a,b}(s)$ and their reciprocals have been investigated and applied in many papers such as~\cite{emv-log-convex-simple.tex, ijmest-bernoulli, mon-element-exp-AIMS.tex, Guo-Qi-Filomat-2011-May-12.tex, TamsuiOxf.J.Math.Sci.(Mar13-07).tex, MMasjedJamei-FengQi-HMSrivastava2008.tex, emv-rs-schur-rocky, 1st-Sirling-Number-2012.tex, pams-62, Bessel-ineq-Dgree-CM.tex, QiBerg.tex, Cheung-Qi-Rev.tex, steffensen-pair-Anal, onsp, Extended-Binet-remiander-comp.tex-Slovaca, simp-exp-degree-ext.tex, frac(bx-ax)x, best-constant-one-simple-real.tex}. This subject was surveyed in~\cite{Qi-Springer-2012-Srivastava.tex}.
Recently, it was discovered that the derivatives of the function $\frac{\sigma(s)}s=\frac1{1-e^{-s}}$ have something to do with Stirling numbers of the first and second kinds in combinatorics and number theory. For detailed information, please refer to~\cite[p.~559]{GKP-Concrete-Math} and~\cite{Eight-Identy-More.tex, exp-derivative-sum-Combined.tex, Exp-Diff-Ratio-Wei-Guo.tex, CAM-D-13-01430-Xu-Cen}.
\par
On logarithmically concave functions, we may establish the following proposition.

\begin{prop}\label{log-concave-mon-lem}
If $f(x)$ is differentiable and logarithmically concave \textup{(}or logarithmically convex respectively\textup{)} on $(-\infty,\infty)$, then the product $f(x)f(\lambda-x)$ for any fixed number $\lambda\in\mathbb{R}$ is increasing \textup{(}or decreasing respectively\textup{)} with respect to $x\in\bigl(-\infty,\frac{\lambda}2\bigr)$ and decreasing \textup{(}or increasing respectively\textup{)} with respect to $x\in\bigl(\frac{\lambda}2,\infty\bigr)$.
\end{prop}

\begin{proof}
Taking the logarithm of $f(x)f(\lambda-x)$ and differentiating give
$$
\{\ln[f(x)f(\lambda-x)]\}'=\frac{f'(x)}{f(x)}-\frac{f'(\lambda-x)}{f(\lambda-x)}.
$$
In virtue of the logarithmic concavity of $f(x)$, it follows that the function $\frac{f'(x)}{f(x)}$ is decreasing and $\frac{f'(\lambda-x)}{f(\lambda-x)}$ is increasing on $(-\infty,\infty)$. From the obvious fact that $\{\ln[f(x)f(\lambda-x)]\}'|_{x=\lambda/2}=0$, it is deduced that $\{\ln[f(x)f(\lambda-x)]\}'<0$ for $x>\frac\lambda2$ and $\{\ln[f(x)f(\lambda-x)]\}'>0$ for $x<\frac\lambda2$. Hence, the function $f(x)f(\lambda-x)$ is decreasing for $x>\frac\lambda2$ and increasing for $x<\frac\lambda2$.
\par
For the case of $f(x)$ being logarithmically convex, it may be proved similarly.
\end{proof}

The techniques in the proof of Proposition~\ref{log-concave-mon-lem} was ever utilized in the papers~\cite{MIA-1729.tex, Cheung-Qi-Rev.tex, notes-best.tex, Gini-Convex-t.tex} and closely related references therein.
\par
By Proposition~\ref{log-concave-mon-lem}, it may be deduced that the function $\sigma''(s)\sigma''(t-s)$ is increasing with respect to $s\in\bigl(0,\frac{t}2\bigr)$ and decreasing with respect to $s\in\bigl(\frac{t}2,t\bigr)$.

\subsection{}
Motivated by Lemma~\ref{convolution-ineq}, the proof of Theorem~\ref{x-4-further-thm}, and Proposition~\ref{log-concave-mon-lem}, we pose the following open problem.

\begin{open}
When $f_i$ for $1\le i\le n$ are all logarithmically concave on $[0,a)$ for all $a>0$, can one find a stronger lower bound than the one in~\eqref{convolution-lower=b} for the convolution $f_1\ast f_2\ast\dotsm\ast f_n(x)$?
\end{open}

\section{Appendix}\label{appendix}
In this section, we would like to spend a lot of texts to detail the proof for the positivity of the function $R(t)$ defined by~\eqref{R(t)-dfn-eq}.

\begin{prop}\label{prop-exp-am}
The function $R(t)$ defined by~\eqref{R(t)-dfn-eq} and all of its derivatives are positive on $(0,\infty)$. In other words, the function $R(t)$ is absolutely monotonic on $(0,\infty)$, or say, the function $R(-t)$ is completely monotonic on $(-\infty,0)$.
\end{prop}

\begin{proof}
With the help of M\textsc{athematica}, consecutive differentiation and straightforward simplification give

and
\begin{gather*}
R^{(k)}(0)=0,\quad 0\le k\le6; \quad R_1^{(k)}(0)=0,\quad 0\le k\le8;\quad 
R_2(0)=0, \quad
R_2'(0)=94594500, \\ 
R_2''(0)=5675670000,\quad R_2^{(3)}(0)=186994407600, \quad
R_2^{(4)}(0)=4870671886080, \\ 
R_2^{(5)}(0)=115977230617248, \quad R_2^{(6)}(0)=2568045873032640, \quad R_2^{(7)}(0)=51468046246929312,\\
R_2^{(8)}(0)=919261925834668416; \quad R_3(0)=2127921124617288,\quad
R_3'(0)=31808262868948566,\\ 
R_3''(0)=422643933990433848,\quad
R_3^{(3)}(0)=5059127238393038844,\quad R_3^{(4)}(0)=55311317976132200928,\\
R_3^{(5)}(0)=559358124246419710080,\quad R_3^{(6)}(0)=5291254778485384410240,\\
R_3^{(7)}(0)=47275711643630010675648,\quad R_3^{(8)}(0)=402327284442724791735744;\\
R_4(0)=25145455277670299483484, \quad R_4'(0)=180175476735597623990958, \\ R_4''(0)=1233364679292092331540240, \quad R_4^{(3)}(0)=8128711084294241029973952, \\ R_4^{(4)}(0)=51948114141680687652017184, \quad  R_4^{(5)}(0)=323999742368987506576507680, \\ R_4^{(6)}(0)=1983712454318091617900255520, \quad  R_4^{(7)}(0)=11984168286345820788017615520, \\ R_4^{(8)}(0)=71752827530169076522475337120;\quad
R_5(0)=448455172063556728265470857, \\ R_5'(0)=2222062109238254700527363412, \quad R_5''(0)=10966377777101822969342094780, \\ R_5^{(3)}(0)=54073563048190450871500371120, \quad R_5^{(4)}(0)=266769748534576065247741326480, \\ R_5^{(5)}(0)=1316887765697877890640258458880, \quad R_5^{(6)}(0)=6499339666080758476651443438720, \\ R_5^{(7)}(0)=32030330270749184706488159641920, \quad R_5^{(8)}(0)=157419994070058290561168403698880;\\
R_6(0)=24993013334803239242170973,\quad  R_6'(0)=97363392902732249529746070,\\ R_6''(0)=376439498975614717003617884,\quad  R_6^{(3)}(0)=1442729528728439890600508052,\\ R_6^{(4)}(0)=5477808342361637707363220348,\quad  R_6^{(5)}(0)=20602476681796080169865157140,\\ R_6^{(6)}(0)=76776923758516348244316687404,\quad  R_6^{(7)}(0)=283621637079942700292207261572,\\ R_6^{(8)}(0)=1039178314812864090945470799068;\quad
R_7(0)=259794578703216022736367699767,\\ R_7'(0)=684880850316419203514174536198,\quad R_7''(0)=1781525699945699167445177663968,\\ R_7^{(3)}(0)=4578635505175348952464533634608,\quad R_7^{(4)}(0)=11640343378188197428722427451648,\\ R_7^{(5)}(0)=29304071938148606553453533106688,\quad R_7^{(6)}(0)=73113605516101190505198264803328,\\ R_7^{(7)}(0)=180921493369266920211167500832768,\quad R_7^{(8)}(0)=444290887449456146125411745425408;\\
R_8(0)=4427324691580199160210177629,\quad R_8'(0)=6367820885649279115456587820,\\ R_8''(0)=8983950893934449574356812800,\quad R_8^{(3)}(0)=12452006779921850992338297600,\\ R_8^{(4)}(0)=16978056111501152993675481600,\quad R_8^{(5)}(0)=22800344803799642868441945600,\\ R_8^{(6)}(0)=30191839866409652462819289600,\quad R_8^{(7)}(0)=39462797393404173379462233600,\\ R_8^{(8)}(0)=50963300758293091764469977600.
\end{gather*}
The proof of Proposition~\ref{prop-exp-am} is thus completed.
\end{proof}

\begin{rem}
This is a corrected and extended version of the preprint~\cite{degree-4-tri-tetra-gamma.tex}.
\end{rem}

\end{document}